\documentclass[11pt]{amsart}

\usepackage{longtable}
\usepackage{enumerate}

\allowdisplaybreaks

\usepackage[english]{babel}

\allowdisplaybreaks

\usepackage[
pdftex,hyperfootnotes]{hyperref}
\hypersetup{
	colorlinks=true,
	linkcolor=RoyalBlue, 
	urlcolor=RoyalPurple,
	citecolor=OliveGreen,
	pdftitle={Oscillory banded matrices and mixed multiple orthogonality},
	bookmarks=true,
	pdfpagemode=FullScreen,
}
\usepackage[usenames,dvipsnames,svgnames,table,x11names]{xcolor}
\usepackage{pgfplots}
\usepackage{tikz}
\usepackage{tikz-3dplot}
\usetikzlibrary{automata,quotes, chains,matrix,calc,shadows,shapes.callouts,shapes.geometric,shapes.misc,positioning,patterns,decorations.shapes,
	decorations.pathmorphing,decorations.markings,decorations.fractals,decorations.pathreplacing,shadings,fadings,arrows.meta,bending}

\usepackage[renew-dots,renew-matrix]{nicematrix}
\usepackage[utf8]{inputenc}

\usepackage{comment}
\usepackage[textwidth=17.5cm,textheight=22.275cm,
height=
24.275cm,
width=18cm]{geometry}

\usepackage{amssymb,latexsym,amsmath,amsthm,bm}
\usepackage{mathrsfs}
\usepackage{mathtools,arydshln,mathdots}

\usepackage{tcolorbox}

\usepackage[table]{xcolor}

\usepackage{Baskervaldx}
\usepackage[]{newtxmath}

\usepackage{bigints}

\theoremstyle{plain}

\newtheorem{teo}{Theorem}[section]

\newtheorem{pro}[teo]{Proposition}

\theoremstyle{remark}
\newtheorem{rem}[teo]{Remark}

\allowdisplaybreaks[1]

\makeatletter
\DeclareRobustCommand{\gaussk}{\DOTSB\gaussk@\slimits@}
\newcommand{\gaussk@}{\mathop{\vphantom{\sum}\mathpalette\bigcal@{K}}}

\newcommand{\bigcal@}[2]{%
	\vcenter{\m@th
		\sbox\z@{$#1\sum$}%
		\dimen@=\dimexpr\ht\z@+\dp\z@
		\hbox{\resizebox{!}{0.8\dimen@}{$\mathcal{K}$}}%
	}%
}
\newcommand{\cfracplus}{\mathbin{\cfracplus@}}
\newcommand{\cfracplus@}{%
	\sbox\z@{$\dfrac{1}{1}$}%
	\sbox\tw@{$+$}%
	\raisebox{\dimexpr\dp\tw@-\dp\z@\relax}{$+$}%
}
\newcommand{\cfracdots}{\mathord{\cfracdots@}}
\newcommand{\cfracdots@}{%
	\sbox\z@{$\dfrac{1}{1}$}%
	\sbox\tw@{$+$}%
	\raisebox{\dimexpr\dp\tw@-\dp\z@\relax}{$\cdots$}%
}
\makeatother

\makeatletter
\newcommand*{\relrelbarsep}{.386ex}
\newcommand*{\relrelbar}{%
	\mathrel{%
		\mathpalette\@relrelbar\relrelbarsep
	}%
}
\newcommand*{\@relrelbar}[2]{%
	\raise#2\hbox to 0pt{$\m@th#1\relbar$\hss}%
	\lower#2\hbox{$\m@th#1\relbar$}%
}
\providecommand*{\rightrightarrowsfill@}{%
	\arrowfill@\relrelbar\relrelbar\rightrightarrows
}
\providecommand*{\leftleftarrowsfill@}{%
	\arrowfill@\leftleftarrows\relrelbar\relrelbar
}
\providecommand*{\xrightrightarrows}[2][]{%
	\ext@arrow 0359\rightrightarrowsfill@{#1}{#2}%
}
\providecommand*{\xleftleftarrows}[2][]{%
	\ext@arrow 3095\leftleftarrowsfill@{#1}{#2}%
}
\makeatother

\newcommand*\pFqskip{8mu}
\catcode`,\active
\newcommand*\pFq{\begingroup
	\catcode`\,\active
	\def ,{\mskip\pFqskip\relax}%
	\dopFq
}
\catcode`\,12
\def\dopFq#1#2#3#4#5{%
	{}_{#1}F_{#2}\biggl[\genfrac..{0pt}{}{#3}{#4};#5\biggr]%
	\endgroup
}

\usepackage{appendix}

\usepackage{xcolor} 

\usepackage[normalem]{ulem}
\usepackage{soul} 

\usepackage{comment} 

\usepackage{tikz}
\usetikzlibrary{shapes,arrows}
\usepackage{verbatim}

\tikzstyle{block} = [draw, rectangle, 
minimum height=3em, minimum width=2em]

\usepackage{mathrsfs} 

\begin{document}

	\title[Bidiagonal factorization of the  multiple Hahn recurrence matrix]{Bidiagonal factorization of the  recurrence matrix for the  Hahn  multiple orthogonal polynomials }
\author[A Branquinho]{Amílcar Branquinho$^{1}$}
\address{$^1$CMUC, Departamento de Matemática,
	Universidade de Coimbra, 3001-454 Coimbra, Portugal}
\email{$^1$ajplb@mat.uc.pt}

\author[JEF Díaz]{Juan EF Díaz$^{2}$}
\address{$^2$CIDMA, Departamento de Matemática, Universidade de Aveiro, 3810-193 Aveiro, Portugal}
\email{$^2$juan.enri@ua.pt}

\author[A Foulquié]{Ana Foulquié-Moreno$^{3}$}
\address{$^3$CIDMA, Departamento de Matemática, Universidade de Aveiro, 3810-193 Aveiro, Portugal}
\email{$^3$foulquie@ua.pt}

\author[M Mañas]{Manuel Mañas$^{4}$}
\address{$^4$Departamento de Física Teórica, Universidad Complutense de Madrid, Plaza Ciencias 1, 28040-Madrid, Spain 
}
\email{$^4$manuel.manas@ucm.es}



\begin{abstract}
This paper explores a factorization using bidiagonal matrices of the recurrence matrix of Hahn multiple orthogonal polynomials. The factorization is expressed in terms of ratios involving the generalized hypergeometric function ${}_3F_2$ and is proven using recently discovered contiguous relations. 
Moreover, employing the multiple Askey scheme, a bidiagonal factorization is derived for the Hahn descendants, including Jacobi-Piñeiro, multiple Meixner (kinds I and II), multiple Laguerre (kinds I and II), multiple Kravchuk, and multiple Charlier, all represented in terms of hypergeometric functions.
For the cases of multiple Hahn, Jacobi-Piñeiro, Meixner of kind II, and Laguerre of kind I, where there exists a region where the recurrence matrix is nonnegative, subregions are identified where the bidiagonal factorization becomes a positive bidiagonal factorization.
\end{abstract}

\maketitle

\section{Introduction}

In this paper, we will deal with tetradiagonal Hessenberg matrices of the type:
\begin{align}\label{eq:laT}
	T \coloneqq
	\begin{bmatrix}
		b_0 & 1 & 0 &\Cdots[shorten-end=9pt]  & \\
		c_1 & b_1 & 1& \Ddots[shorten-end=9pt] & \\
		d_2 & c_2 & b_2 & \Ddots[shorten-end=9pt] & \\[2pt]
		0 & d_3 & c_3 & b_3 & \\
		\Vdots[shorten-end=7pt] & \Ddots[shorten-end=-5pt]  & \Ddots[shorten-end=-5pt]  & \Ddots[shorten-end=-5pt]  & \Ddots[shorten-end=10pt] 
	\end{bmatrix},
\end{align}
and  its leading  submatrices, also known as truncations:
\begin{align}\label{eq:laT truncada}
	T_m &\coloneqq
	\begin{bmatrix}
		b_0 & 1 & 0 & \Cdots & & & 0\\
		c_1 & b_1 & 1 & \Ddots & & & \Vdots\\[5pt]
		d_2 & c_2 & b_2 & \Cdots & & & \\
		0 & d_3 & c_3 &\Ddots & \Ddots& & \\
		\Vdots & \Ddots[shorten-end=2pt] &\Ddots[shorten-end=-2pt] & \Ddots[shorten-end=1pt]  & & & 0\\
		& & & & & & 1\\
		0 & \Cdots & & 0& d_{m-1} & c_{m-1} & b_{m-1}
	\end{bmatrix},
	& m\in &\{1,\dots,N+1\}.
\end{align}
In this work, we focus on tetradiagonal Hessenberg matrices, which originate from the theory of multiple orthogonal polynomials \cite{nikishin_sorokin, Ismail, afm}. These matrices are used to represent recurrence relations of the following form:
\begin{align*}
	T
	\begin{bmatrix}
		B^{(0)}(x)\\ B^{(1)}(x)\\ B^{(2)}(x)\\ \vdots
	\end{bmatrix}
	&= x
	\begin{bmatrix}
		B^{(0)}(x)\\ B^{(1)}(x)\\ B^{(2)}(x)\\ \vdots
	\end{bmatrix},
\end{align*}
where $B^{(n)}(x)$ denotes multiple orthogonal polynomials with respect to two weights for $n \in \mathbb{N}_0$ (on the step-line). Our particular focus lies on tetradiagonal Hessenberg matrices that represent the recurrence relations for multiple orthogonal polynomials of Hahn, Meixner (kinds I and II), Kravchuk, Jacobi-Piñeiro, Laguerre (kinds I and II), and Charlier. These multiple orthogonal polynomials are descendants of the Hahn multiple orthogonal polynomials in the Askey scheme \cite{AskeyII}.
\begin{center}
	\begin{tikzpicture}[node distance=3cm]
		\node[fill=SkyBlue!65,block] (a) {Jacobi--Piñeiro}; 
		\node[ fill=SkyBlue!65,block, right of=a] (b) {Meixner II};
		\node[ fill=SkyBlue!65,block, right of = b] (c) {Meixner I};
		\node[ fill=SkyBlue!65,block, right of=c] (f) {Kravchuk}; 
		\node[ fill=SkyBlue!65,block,above of = c, left =.85cm] (d) {Hahn};
		\node[fill=SkyBlue!65,block, below of = b] (e) {Laguerre I};
		\node[fill=SkyBlue!65,block, below of = a] (g) {Laguerre II};
		\node[ fill=SkyBlue!65,block, below of =f] (i) {Charlier};
		\node[ fill=Red!55, block, below of =c] (j) {Hermite};
		\draw[-latex] (d)--(a); 
		\draw[-latex] (d)--(b);
		\draw[-latex] (d)--(c);
		\draw[-latex] (d)--(f);
		\draw[-latex] (a)--(e);
		\draw[-latex] (b)--(e);
		\draw[-latex] (c)--(i);
		\draw[-latex] (c)--(g);
		\draw[-latex] (f)--(i);
		\draw[-latex] (a)--(g);
		\draw[dashed,->] (e)--(j);
		\draw[dashed,->] (a)--(j);
		\draw[dashed,->] (f)--(j);
		\draw[dashed,->] (i)--(j);
		
		\draw (4.5,-4) node
	{\begin{minipage}{0.5\textwidth}
			\begin{center}\small
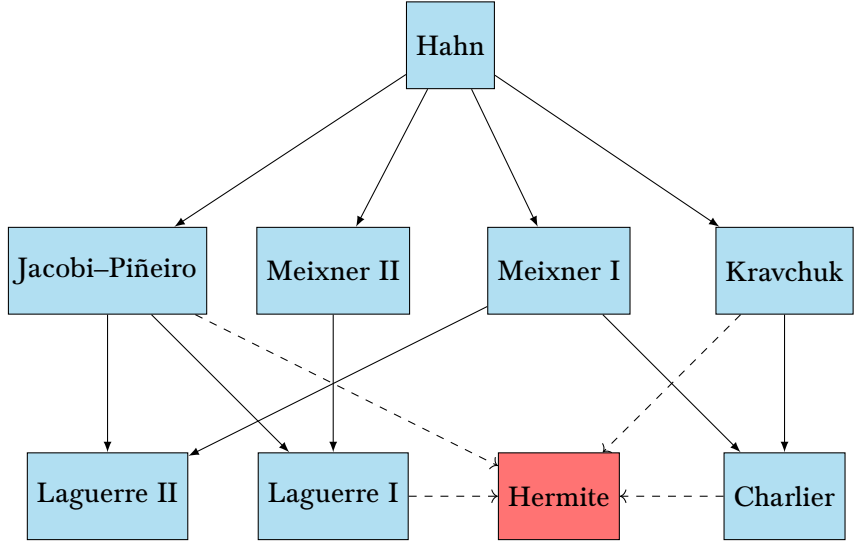
				\textbf{Descendants of Hahn in the multiple Askey scheme }
			\end{center}
	\end{minipage}};
	\end{tikzpicture}
\end{center}

\enlargethispage{1cm}
The goal of this paper is to provide for such tetradiagonal Hessenberg matrices bidiagonal factorizations of the form:
\begin{align}
	\label{factorizacionbidiagonalmultiple}
	T_{m}&=L_{1,m}L_{2,m}U_{m}, &
	\ L_{1,m}&\coloneq
	\begin{bNiceMatrix}[small]
		1 & 0 & & \cdots & & 0\\
		a_2 & 1 & \Ddots [shorten-end=-2pt] & & & \vdots\\
		0 & a_5 & \ddots& & & \\
		\vdots & \ddots & \Ddots[shorten-end=-5pt]& & & \\
		& & & & & 0\\[6pt]
		0 & \cdots & & 0 & a_{3m-4} & 1
	\end{bNiceMatrix},
	&L_{2,m}&\coloneq
	\begin{bNiceMatrix}[small]
		1 & 0 & & \cdots & & 0\\
		a_3 & 1 & \Ddots [shorten-end=-2pt]  & & & \vdots\\
		0 & a_6 & \ddots& & & \\
		\vdots & \ddots & \Ddots[shorten-end=-5pt]& & & \\
		& & & & & 0\\[6pt]
		0 & \cdots & & 0 & a_{3m-3} & 1
	\end{bNiceMatrix},&
	U_{m}&\coloneq\begin{bNiceMatrix}[small]
		a_1 & 1 & 0 & \cdots & & 0\\
		0 & a_4 & 1 & \Ddots[shorten-end=0pt] & & \vdots\\
		\vdots & \ddots& \Ddots[shorten-end=-7pt]& \ddots & & \\
		& & & & & 0\\[6pt]
		& & & & & 1\\[6pt]
		0 & \cdots& & & 0 & a_{3m-2}
	\end{bNiceMatrix}.
\end{align}

In \cite{genetico}, a bidiagonal factorization was discovered for the recurrence matrix of the Jacobi--Piñeiro multiple orthogonal polynomials. Similarly, in \cite{Lima-Loureiro}, a similar bidiagonal factorization was found for the recurrence matrix of the hypergeometric multiple orthogonal polynomials. The existence of a positive bidiagonal factorization (PBF) was guaranteed when the parameters defining the weights fell within a specific region known as the semi-band.

Additionally, in \cite{stbbm0} and \cite{stbbm}, a Favard-type theorem was proven, ensuring that for a given banded Hessenberg matrix or a banded matrix where a positive bidiagonal factorization holds ($a_n>0$, $n\in\mathbb{N}$), the existence of a set of multiple orthogonal polynomials or mixed multiple orthogonal polynomials, respectively, can be ensured.

Moreover, in \cite{Hypergeometric}, \cite{JP}, and \cite{finite-Markov}, it was discovered that these PBFs are relevant for finding associated Markov chains whose stochastic matrix can be factorized in terms of stochastic matrices describing pure birth or pure death Markov chains. The bidiagonal factorization of tetradiagonal matrices is associated with corresponding Darboux transformations and Christoffel perturbations, as described in \cite{Darboux}. Additionally, the PBF of tetradiagonal matrices leads to a very rich theory involving the theory of continued fractions \cite{PBF_tetra}.

In \cite{HahnI}, we provided hypergeometric expressions for the Hahn multiple orthogonal polynomials of type I, which allowed for the derivation of explicit hypergeometric expressions for all type I multiple orthogonal polynomials that are descendants in the multiple Askey scheme of the Hahn multiple orthogonal polynomials. These descendants include Jacobi-Piñeiro, multiple Laguerre of the first and second kind, multiple Meixner of the first and second kind, multiple Charlier, and multiple Kravchuk.

In this paper, we utilize some of the ideas learned in \cite{HahnI} to find a bidiagonal factorization of the recurrence matrix for Hahn multiple orthogonal polynomials. Additionally, when this matrix is nonnegative, we identify a region of parameters where a PBF exists. Furthermore, using the multiple Askey scheme, we present the bidiagonal factorization for all the aforementioned descendants of the Hahn multiple orthogonal polynomials. For cases where the recurrence tetradiagonal Hessenberg matrix is nonnegative, such as Jacobi-Piñeiro, Laguerre of the first kind, and Meixner of the second kind, we determine the corresponding PBF.

There is a Mathematica notebook written to compute bidiagonal factorizations of the recurrence matrix for Hahn multiple orthogonal polynomials and its descendants in the multiple Asket scheme. It has been  uploaded to the  \hyperref{https://notebookarchive.org}{}{}{Mathematica Notebook Archive} and to  \hyperref{https://github.com/ManuelManas/Bidiagonal_factorization/blob/main/}{}{}{GitHub}.

\section{Bidiagonal Factorization for the multiple Hahn case}

This bidiagonal factorization \eqref{factorizacionbidiagonalmultiple} is not unique and depends on the choice of the initial condition $a_2$. The best choice we have found   \cite[Theorem 2]{stbbm0} is $a_2=-\frac{A_{(1,0),1}(x)}{A_{(1,1),1}(x)}$, as it leads to the closed formulas.

The coefficients of the recurrence relations \eqref{eq:laT} and \eqref{eq:laT truncada} can be found in \cite[\S 4.5]{Arvesu} the recursion coeffcients $b_n(\alpha_1,\alpha_2,\beta,N)$, $c_n(\alpha_1,\alpha_2,\beta,N)$ and $d_n(\alpha_1,\alpha_2,\beta,N)$ are 
{\small\begin{align}\label{recurrenciaHahnmultiple}
 \left\{\begin{aligned}
 	 b_{2m}&=\begin{multlined}[t][\textwidth]
 A(m,m,\alpha_1,\alpha_2,\beta,N)+A(m,m,\alpha_2,\alpha_1+1,\beta,N)+C(m,m+1,\alpha_1,\alpha_2,\beta,N)+D(m,m,\alpha_1,\alpha_2,\beta,N),
 \end{multlined}\\
 b_{2m+1}&=\begin{multlined}[t][\textwidth]
 	A(m,m+1,\alpha_2,\alpha_1,\beta,N)+A(m+1,m,\alpha_1,\alpha_2+1,\beta,N)
 +C(m+1,m+2,\alpha_2,\alpha_1,\beta,N)+D(m,m+1,\alpha_2,\alpha_1,\beta,N),
 \end{multlined}\\
 c_{2m}&=\begin{multlined}[t][.95\textwidth]
 	(A(m,m,\alpha_1,\alpha_2,\beta,N)+A(m,m,\alpha_2,\alpha_1+1,\beta,N)\\+D(m,m,\alpha_1,\alpha_2,\beta,N)) {C}(m,m+1,\alpha_2,\alpha_1,\beta,N)
 +A(m,m,\alpha_1,\alpha_2,\beta,N)B(m,m,\alpha_1,\alpha_2,\beta,N),
 \end{multlined}\\
 c_{2m+1}&=\begin{multlined}[t][.95\textwidth]
 	(A(m,m+1,\alpha_2,\alpha_1,\beta,N)+A(m+1,m,\alpha_1,\alpha_2+1,\beta,N)
 \\
 +D(m,m+1,\alpha_2,\alpha_1,\beta,N)){C}(m+1,m+1,\alpha_1,\alpha_2,\beta,N)
+A(m,m+1,\alpha_2,\alpha_1,\beta,N)B(m,m+1,\alpha_2,\alpha_1,\beta,N),
 \end{multlined}\\
 d_{2m}&=A(m,m,\alpha_1,\alpha_2,\beta,N)B(m,m,\alpha_1,\alpha_2,\beta,N)C(m,m,\alpha_1,\alpha_2,\beta,N),\\
 d_{2m+1}&=A(m,m+1,\alpha_2,\alpha_1,\beta,N)B(m,m+1,\alpha_2,\alpha_1,\beta,N)C(m,m+1,\alpha_2,\alpha_1,\beta,N).
 \end{aligned}\right.
\end{align}
with the coefficients $	A(n_1,n_2,\alpha_1,\alpha_2,\beta,N)$, 	$B(n_1,n_2,\alpha_1,\alpha_2,\beta,N)$,	$C(n_1,n_2,\alpha_1,\alpha_2,\beta,N)$ and 	$D(n_1,n_2,\alpha_1,\alpha_2,\beta,N)$ given by
{\begin{align}
\label{coefsauxiliares} 
\left\{\begin{aligned}
	A&=\dfrac{n_1(n_1+n_2+\alpha_2+\beta)(n_1+n_2+\beta)(N+n_1+\alpha_1+\beta+1)}{(n_1+2n_2+\alpha_2+\beta)(2n_1+n_2+\alpha_1+\beta)(2n_1+n_2+\alpha_1+\beta+1)},\\
 B&=\dfrac{(n_1+\alpha_1-\alpha_2)(n_1+n_2+\alpha_1+\beta)(n_1+n_2+\beta-1)(N-n_1-n_2+1)}{(n_1+2n_2+\alpha_2+\beta-1)(2n_1+n_2+\alpha_1+\beta)(2n_1+n_2+\alpha_1+\beta-1)},\\
 C&=\dfrac{(n_1+\alpha_1)(n_1+n_2+\alpha_1+\beta-1)(n_1+n_2+\alpha_2+\beta-1)(N-n_1-n_2+2)}{(n_1+2n_2+\alpha_2+\beta-2)(2n_1+n_2+\alpha_1+\beta-2)(2n_1+n_2+\alpha_1+\beta-1)},\\
 D&=\dfrac{n_1n_2(n_1+n_2+\beta)}{(2n_1+n_2+\alpha_1+\beta+1)(n_1+2n_2+\alpha_2+\beta)}.
\end{aligned}\right.
\end{align}}

\clearpage
\begin{teo}
\label{Pro:PBF_Hahn}
   The  recurrence matrix, with coefficients as in \eqref{recurrenciaHahnmultiple}, satisfies a bidiagonal factorization \eqref{factorizacionbidiagonalmultiple}  respect to the following coefficients $a_n(\alpha_1,\alpha_2,\beta,N)$

\begin{align}
\label{aHahnmultiple}
\left\{\begin{aligned}
	    a_{6n+1}&=
    \dfrac{(N-2n)(\alpha_1+1+n)(\alpha_1+\beta+2n+1)(\alpha_2+\beta+2n+1)}{(\alpha_1+\beta+3n+1)_2(\alpha_2+\beta+3n+1)}\\
   a_{6n+4}&=\dfrac{(N-2n-1)(\alpha_2+1+n)(\alpha_1+\beta+2n+2)(\alpha_2+\beta+2n+2)}{(\alpha_1+\beta+3n+3)(\alpha_2+\beta+3n+2)_2},\\
a_{6n+2}&=\begin{multlined}[t][.7\textwidth]\dfrac{(N-2n)(n)_n(\beta+2n+1)(\alpha_2-\alpha_1+n)(\alpha_2+\beta+n+1)}{(n+1)_n(\alpha_1+\beta+3n+2)(\alpha_2+\beta+3n+1)_2}
    \dfrac{\pFq{3}{2}{-n,-N,\alpha_2-\alpha_1-n}{-2n+1,\alpha_2+\beta+n+1}{1}}{\pFq{3}{2}{-n,-N,\alpha_2-\alpha_1-n}{-2n,\alpha_2+\beta+n+2}{1}},\end{multlined}\\
     a_{6n+5}&=\begin{multlined}[t][.8\textwidth]\dfrac{(n+1)(N-2n-1)(\beta+2n+2)(\alpha_1-\alpha_2+n+1)(\alpha_1+\beta+2+n+N)}{(2n+1)(\alpha_1+\beta+3n+3)_2(\alpha_2+\beta+3n+3)}
     	\dfrac{\pFq{3}{2}{-n,-N,\alpha_2-\alpha_1-n}{-2n,\alpha_2+\beta+n+2}{1}}{\pFq{3}{2}{-n-1,-N,\alpha_2-\alpha_1-n-1}{-2n-1,\alpha_2+\beta+n+2}{1}},\end{multlined}\\
    a_{6n+3}&=\begin{multlined}[t][.8\textwidth]\dfrac{(2n+1)(\beta+2n+1)(\alpha_1+\beta+2n+2)(\alpha_2+\beta+2n+2)}{(\alpha_1+\beta+3n+2)_2(\alpha_2+\beta+3n+2)}
    \dfrac{\pFq{3}{2}{-n-1,-N,\alpha_2-\alpha_1-n-1}{-2n-1,\alpha_2+\beta+n+2}{1}}{\pFq{3}{2}{-n,-N,\alpha_2-\alpha_1-n}{-2n,\alpha_2+\beta+n+2}{1}},\end{multlined}\\
     a_{6n+6}&=\begin{multlined}[t][.8\textwidth]\dfrac{2(n+1)(\beta+2n+2)(\alpha_1+\beta+2n+3)(\alpha_2+\beta+2n+3)(\alpha_2+\beta+2+n+N)}{(\alpha_1+\beta+3n+4)(\alpha_2+\beta+3n+3)_2(\alpha_2+\beta+n+2)}
     	\dfrac{\pFq{3}{2}{-n-1,-N,\alpha_2-\alpha_1-n-1}{-2n-2,\alpha_2+\beta+n+3}{1}}{\pFq{3}{2}{-n-1,-N,\alpha_2-\alpha_1-n-1}{-2n-1,\alpha_2+\beta+n+2}{1}}.\end{multlined}
\end{aligned}\right.
\end{align}
This is a positive
bidiagonal factorization in the ``semi-band'' whenever $-1<\alpha_1-\alpha_2<0$.
\end{teo}

To accomplish this, we need to handle the $_3F_2$ hypergeometric functions. For that purpose, we introduce the following theorem, originally established by Ebisu and Iwasaki in \cite[Theorem 1.1]{contiguas}.

\begin{teo}[Ebisu--Iwasaki]
	\label{teoremaEbisuIwasaki}
	Let $\vec{a}=(a_0,a_1,a_2,a_3,a_4)\in\mathbb{R}^5$ and $\vec{p}=(p_0,p_1,p_2,p_3,p_4),\vec{q}=(q_0,q_1,q_2,q_3,q_4)\in\mathbb{Z}^5$ be distinct shift vectors. Then, there exist unique rational functions $u,v$ such that
	\begin{align*}
		\pFq{3}{2}{a_0,a_1,a_2}{a_3,a_4}{1}
		=&u(\vec{a})\frac{\Gamma(a_3)\Gamma(a_4)}{\Gamma(a_0)\Gamma(a_1)\Gamma(a_2)}\frac{\Gamma(a_0+p_0)\Gamma(a_1+p_1)\Gamma(a_2+p_2)}{\Gamma(a_3+p_3)\Gamma(a_4+p_4)}\pFq{3}{2}{a_0+p_0,a_1+p_1,a_2+p_2}{a_3+p_3,a_4+p_4}{1}\\
		&+v(\vec{a})\frac{\Gamma(a_3)\Gamma(a_4)}{\Gamma(a_0)\Gamma(a_1)\Gamma(a_2)}\frac{\Gamma(a_0+q_0)\Gamma(a_1+q_1)\Gamma(a_2+q_2)}{\Gamma(a_3+q_3)\Gamma(a_4+q_4)}\pFq{3}{2}{a_0+q_0,a_1+q_1,a_2+q_2}{a_3+q_3,a_4+q_4}{1}
	\end{align*}
	and there is a systematic recipe to determine $u,v$ in a finite number of steps.
\end{teo}

The steps to find $u,v$ are given in \cite[Recipe 5.4]{contiguas}. They can be obtained through the expressions
\begin{align}
	\label{coefsrelacioncontiguidad}   
	u(\vec{a})=\frac{A(\vec{a},\vec{q})_{1,2}}{\det{A(\vec{a},\vec{p})A(\vec{a}+\vec{p},\vec{q}-\vec{p})_{1,2}}}, \quad v(\vec{a})=-\frac{A(\vec{a},\vec{p})_{1,2}}{\det{A(\vec{a},\vec{p})A(\vec{a}+\vec{p},\vec{q}-\vec{p})_{1,2}}}
\end{align}
where $A(\vec{a},\vec{p})$ (similarly for $A(\vec{a},\vec{q})$ and $A(\vec{a}+\vec{p},\vec{q}-\vec{p})$) is a product of contiguous $2\times 2$ matrices given by
\begin{align}\label{contiguousproduct}   
	A(\vec{a},\vec{p}):=A^{\epsilon_m}_{i_m}(\vec{a}+\vec{p}_{m-1})A^{\epsilon_{m-1}}_{i_{m-1}}(\vec{a}+\vec{p}_{m-2})\cdots A^{\epsilon_2}_{i_2}(\vec{a}+\vec{p}_1) A^{\epsilon_1}_{i_1}(\vec{a}+\vec{p}_0),
\end{align}
where $\epsilon_{1},\dots,\epsilon_{m}\in\{+,-\}$, $i_1,\dots i_m\in\{0,1,2,3,4\}$ and
\begin{align*}
	\vec{p}_0:=\vec{0}, \quad
	\vec{p}_{k}:=\vec{p}_{k-1}+\epsilon_{k}\vec{e}_{i_k}, \quad k\in\{1,\cdots,m\},
\end{align*}
with 
\begin{gather*}
\begin{aligned}
	\vec{e}_0&\coloneq\begin{bNiceMatrix}1 & 0 &0 &0 &0 \end{bNiceMatrix}, &
\vec{e}_1&\coloneq\begin{bNiceMatrix}0 & 1&0 &0 &0\end{bNiceMatrix}, 
\end{aligned}\\
\begin{aligned}
	\vec{e}_2&\coloneq\begin{bNiceMatrix}0& 0&1 &0&0 \end{bNiceMatrix}, &
\vec{e}_3&\coloneq \begin{bNiceMatrix}0&0 &0 &1 &0\end{bNiceMatrix},&
\vec{e}_4&\coloneq \begin{bNiceMatrix}0 &0 &0 &0 &1\end{bNiceMatrix}.
\end{aligned}
\end{gather*}

This means that the matrix $A^{\epsilon_k}_{i_k}(\vec{a}+\vec{p}_{k-1})$ corresponds to the operation of adding ($\epsilon_{k}=+$) or subtracting ($\epsilon_k=-$) one unit to the $i_k$-th component of the vector $\vec{a}+\vec{p}_{k-1}$. In other words, we have a matrix for each time we need to add or subtract one unit to transition from $\vec{a}$ to $\vec{a}+\vec{p}$. All these matrices $A^{\epsilon}_i$ are given in \cite[Table 2]{contiguas}
\begin{align}
	\label{A012+}
	A_i^+(\vec{a})&=\begin{bmatrix}
		a_i&1\\[4pt] \dfrac{\phi_3(\vec{a})}{s(\vec{a})-2}& \dfrac{a_ja_k-(a_i-a_3+1)(a_i-a_4+1)}{s(\vec{a})-2}
	\end{bmatrix}, &i&\in\{0,1,2\},\\
	\label{A34+}
	A_i^+(\vec{a})&=\dfrac{1}{(a_i-a_0)(a_i-a_1)(a_i-a_2)}\begin{bmatrix}
		a^2_i-\phi_1(\vec{a})a_i+\phi_2(\vec{a})&-s(\vec{a})+1\\[4pt] -\phi_3(\vec{a}) & a_i(s(\vec{a})-1)
	\end{bmatrix},  &i&\in\{3,4\},\\
	\label{A012-}
	A_i^-(\vec{a})&=\dfrac{1}{(a_i-a_3)(a_i-a_4)}\begin{bmatrix}
		\dfrac{(a_i-a_3)(a_i-a_4)-a_ja_k}{a_i-1}&\dfrac{s(\vec{a})-1}{a_i-1}\\[10pt]
		a_ja_k&-s(\vec{a})+1
	\end{bmatrix}, &i&=\{0,1,2\},\\
	\label{A34-}
	A_i^-(\vec{a})&=\begin{bmatrix}
		a_i-1 & 1\\[4pt]
		\dfrac{\phi_3(\vec{a})}{s(\vec{a})-2} & \dfrac{(a_i-1)^2-\phi_1(\vec{a})(a_i-1)+\phi_2(\vec{a})}{s(\vec{a})-2}
	\end{bmatrix}, &i&\in\{3,4\},
\end{align}
where
\begin{align*}
	\phi_1(\vec{a})&\coloneqq a_0+a_1+a_2,&
	\phi_2(\vec{a})&\coloneqq a_0a_1+a_1a_2+a_2a_0,&
	\phi_3(\vec{a})&\coloneqq a_0a_1a_2,&
	s(\vec{a})&\coloneqq a_3+a_4-a_0-a_1-a_2,
\end{align*}
and $	a_ja_k\coloneqq a_1a_2\delta_{i,0}+a_2a_0\delta_{i,1}+a_0a_1\delta_{i,2}$.
With these tools at our disposal, we have all the necessary means to calculate the functions $u$ and $v$ for a given $\vec{a}$, $\vec{p}$, and $\vec{q}$ using the recipe provided earlier. Now, let us proceed to prove Proposition \ref{Pro:PBF_Hahn}.
\begin{proof}
In Theorem \ref{Pro:PBF_Hahn}, proving the matrix equation $T_m=L_{1,m}L_{2,m}U_m$ is equivalent to proving the following system of equations:
	\begin{align}
		\label{a1}
		a_1=b_0\\
		\label{a2}
		a_{6n+2}+a_{6n+3}+a_{6n+4}=b_{2n+1}\\
		\label{a3}
		a_{6n+5}+a_{6n+6}+a_{6n+7}=b_{2n+2}\\
		\label{a4}
		a_1(a_2+a_3)=c_1\\
		\label{a5}
		a_{6n+3}a_{6n+5}+a_{6n+4}(a_{6n+5}+a_{6n+6})=c_{2n+2}\\
		\label{a6}
		a_{6n+6}a_{6n+8}+a_{6n+7}(a_{6n+8}+a_{6n+9})=c_{2n+3}\\
		\label{a7}
		a_{6n+1}a_{6n+3}a_{6n+5}=d_{2n+2}\\
		\label{a8}
		a_{6n+4}a_{6n+6}a_{6n+8}=d_{2n+3}
	\end{align}
The initial conditions \eqref{a1} and \eqref{a4} can be easily verified.

It is worth noting that in Theorem \ref{Pro:PBF_Hahn} the coefficients $a_{6n+1}$ and $a_{6n+4}$ are the only ones that do not involve the $_3F_2$ hypergeometric functions, while the other coefficients do. However, the ratios of $_3F_2$ functions in $a_{6n+3}$ and $a_{6n+6}$ are reciprocals of those in $a_{6n+5}$ and $a_{6n+8}$, respectively. Therefore, equations \eqref{a7} and \eqref{a8} do not rely on the $_3F_2$ functions and can be easily verified.
	
Moreover, due to the same reason, there is no dependence on $_3F_2$ functions in the first term on the left-hand side of equations \eqref{a5} and \eqref{a6}. Now, assuming that \eqref{a2} and \eqref{a3} hold, we can substitute
\begin{align*}
	a_{6n+5}+a_{6n+6}&=b_{2n+2}-a_{6n+7},&
	a_{6n+8}+a_{6n+9}&=b_{2n+3}-a_{6n+10},
\end{align*}
respectively, into equations \eqref{a5} and \eqref{a6}. This yields expressions free of $_3F_2$ functions, which can be easily verified.
	
So, to check Theorem \ref{Pro:PBF_Hahn}, we only need to prove equations \eqref{a2} and \eqref{a3}. Let's begin with \eqref{a2}. Observe that the $_3F_2$ function in the denominator of $a_{6n+2}$ and $a_{6n+3}$ is the same, so it can be factored out. By substituting the coefficients and simplifying, we can rewrite \eqref{a2} as the following contiguous expression:

{\small	\begin{multline}
	\label{a2+a3}
	\Bigg(n(n+1)+\dfrac{n(\alpha_1+\beta+2n+1)(\alpha_2+\beta+n+N+1)}{(\alpha_2+\beta+3n+1)}+
	\dfrac{(n+1)(\alpha_2+\beta+2n+2)(\alpha_1+\beta+n+N+2)}{(\alpha_1+\beta+3n+3)}\Bigg)	\pFq{3}{2}{-n,-N,\alpha_2-\alpha_1-n}{-2n,\alpha_2+\beta+n+2}{1}
		\\=	\dfrac{(N-2n)(n)_n(\alpha_2-\alpha_1+n)(\alpha_2+\beta+n+1)}{(n+1)_n(\alpha_2+\beta+3n+1)}
		\pFq{3}{2}{-n,-N,\alpha_2-\alpha_1-n}{-2n+1,\alpha_2+\beta+n+1}{1} \\
		+\dfrac{(2n+1)(\alpha_1+\beta+2n+2)(\alpha_2+\beta+2n+2)}{(\alpha_1+\beta+3n+3)}
		\pFq{3}{2}{-n-1,-N,\alpha_2-\alpha_1-n-1}{-2n-1,\alpha_2+\beta+n+2}{1}.
\end{multline}}

However, since these shifted $_3F_2$ functions are evaluated at $1$, we know that the coefficients relating them are unique according to Theorem \ref{teoremaEbisuIwasaki}. Therefore, we can calculate these coefficients using the algorithm described in the theorem and confirm that they are the same as those in the previous expression \eqref{a2+a3} derived from Theorem \ref{Pro:PBF_Hahn}. To do this, let's use the notation from Theorem \ref{teoremaEbisuIwasaki} and identify
	\begin{align*}
		\vec{a}&=-n \vec e_0 -N \vec e_1+(\alpha_2-\alpha_1-n) \vec e_2 -2n\vec e_3+ (\alpha_2+\beta+n+2)\vec e_4,&
		\vec{p}&=-\vec e_0 -\vec e_2 -\vec e_3 , &
		\vec{q}&=\vec e_3 -\vec e_4.
	\end{align*}
The matrices we need to compute $u,v$ according to \eqref{coefsrelacioncontiguidad} and \eqref{contiguousproduct} are:
\begin{align*}
	A(\vec{a},\vec{p}) &= A^{-}_3\big(\vec{a}	-\vec e_0-\vec e_2 \big)A^{-}_2\big(\vec{a}-\vec e_0\big)A^{-}_0\big(\vec{a}\big),\\
	A(\vec{a},\vec{q}) &= A^{-}_4\big(\vec{a}+\vec e_3\big)A^{+}_3\big(\vec{a}\big),\\
	A(\vec{a}+\vec{p},\vec{q}-\vec{p}) &= A_{4}^{-}\big(\vec{a}+\vec{p}+
(\vec e_0+\vec e_2+2\vec e_3)	\big)A_{3}^{+}\big(\vec{a}+\vec{p}
	+(\vec e_0+\vec e_2+\vec e_3)\big)A^{+}_3\big(\vec{a}+\vec{p}+(\vec e_0+\vec e_2)
\big) A^{+}_2\big(\vec{a}+\vec{p}+\vec e_0\big)A^{+}_0\big(\vec{a}+\vec{p}\big).
\end{align*}
By computing these matrices using the expressions \eqref{A012+}-\eqref{A34-} for $A^\epsilon_i$, we find that
{\small	\begin{align*}
		A(\vec{a},\vec{p})_{1,2}&=\dfrac{(\alpha_1+\beta+3n+3)(\alpha_1+\beta+n+N+1)}{(n+1)(\alpha_1-\alpha_2+n+1)(\alpha_1+\beta+2n+2)(\alpha_2+\beta+2n+2)},\\
		A(\vec{a},\vec{q})_{1,2}&=\dfrac{(\alpha_2+\beta+3n+1)(\alpha_1+\beta+n+N+1)}{n(N-2n)(\alpha_1-\alpha_2-n)},\\
		A(\vec{a}+\vec{p},\vec{q}-\vec{p})_{1,2}&=
		\begin{multlined}[t][.5\textwidth]
			-\dfrac{1}{n (N-2n) (N-2n-1) ({\alpha_1}-{\alpha_2}-n)}\big(
		{\alpha_1}^2 n ({\alpha_2}+\beta+n+N+1)
	+{\alpha_1} \big({\alpha_2}^2 (n+1)+{\alpha_2} (\beta(4n+2)\\+n(11n+13)+3)+\beta^2(3 n+1)
	+\beta(n(13n+2N+15)+3)+n(n(14 n+5N+27)+4N+15)+2\big)\\
		+{\alpha_2}^2 (n+1) (\beta+n+N+2)+{\alpha_2}\big(\beta^2 (3n+2)+\beta (n(13n+19)+7)+(n+1)(2n+1)(7n+6)\big)\\+{\alpha_2}(n+1)N(2\beta+5n+3)
	+(\beta+3n+2) \big(\beta^2 (2n+1)+\beta\big(7n^2+2n (N+5)+N+3\big)\\+(n+1) (n (7 n+4 N+9)+N+2)\big)\big),
		\end{multlined}\\
		\det A(\vec{a},\vec{p})&=\dfrac{(N-2n-1) (\alpha_1+\beta+n+N+1)}{(n+1) (\alpha_1-\alpha_2+n+1) (\alpha_1+\beta+2n+2) (\alpha_2+\beta+2n+2)}.
	\end{align*}}
Finding $u$ and $v$ with \eqref{coefsrelacioncontiguidad} and multiplying by the corresponding quotients of gamma functions as indicated in Theorem \ref{teoremaEbisuIwasaki}, we find that the coefficients connecting these $_3F_2$ functions are exactly those in \eqref{a2+a3}.

The reasoning is completely analogous for equation \eqref{a3}. As before, the $_3F_2$ function in the denominator of $a_{6n+5},a_{6n+6}$ can be taken as a common factor. By replacing the coefficients $\{b_n\},\{a_n\}$ in Theorem \ref{Pro:PBF_Hahn} and simplifying, we find that
	\begin{multline}\label{hola} 
		\pFq{3}{2}{-n-1,-N,\alpha_2-\alpha_1-n-1}{-2n-1,\alpha_2+\beta+n+2}{1}\\
		=\Biggl(\dfrac{(N-2n-1)(\alpha_1-\alpha_2+n+1)(\alpha_1+\beta+2+n+N)}{(2n+1)(\alpha_1+\beta+3n+3)}
		\pFq{3}{2}{-n,-N,\alpha_2-\alpha_1-n}{-2n,\alpha_2+\beta+n+2}{1}\\
		+\dfrac{2(\alpha_1+\beta+2n+3)(\alpha_2+\beta+2n+3)(\alpha_2+\beta+2+n+N)}{(\alpha_2+\beta+3n+4)(\alpha_2+\beta+n+2)}
		\pFq{3}{2}{-n-1,-N,\alpha_2-\alpha_1-n-1}{-2n-2,\alpha_2+\beta+n+3}{1}\Biggl)\\
		\times\Biggl((n+1)+\dfrac{(\alpha_2+\beta+2n+2)(\alpha_1+\beta+n+N+2)}{(\alpha_1+\beta+3n+3)}+\dfrac{(\alpha_1+\beta+2n+3)(\alpha_2+\beta+n+N+2)}{(\alpha_2+\beta+3n+4)}\Biggl)^{-1}.
	\end{multline}

In this case, we identify
\begin{align*}
	\vec{a}&=-(n+1)\vec e_0-N\vec e_1+(\alpha_2-\alpha_1-n-1)\vec e_2-(2n+1)\vec e_3+(\alpha_2+\beta+n+2)\vec e_4,&
	\vec{p}&=\vec e_0+\vec e_2+\vec e_3,&
	\vec{q}&=-\vec e_3+\vec e_4.
\end{align*}
We can observe that $\vec{p}$ and $\vec{q}$ are exactly the opposite of the ones in the previous case. Now we have
{	\small\begin{align*}
		A(\vec{a},\vec{p})_{1,2}&=-\dfrac{\alpha_1+\beta+3n+3}{N-2n-1},\\
		A(\vec{a},\vec{q})_{1,2}&=\dfrac{(\alpha_2+\beta+3n+4)(\alpha_1+\beta+n+N+2)}{(\alpha_1+\beta+2n+3)(\alpha_2+\beta+2n+3)(\alpha_2+\beta+n+N+2)},\\
		A(\vec{a}+\vec{p},\vec{q}-\vec{p})_{1,2}&=\begin{multlined}[t][.8\textwidth]
			-\dfrac{({\alpha_1}+\beta+n+N+1) }{({\alpha_1}-{\alpha_1}+n+1) ({\alpha_1}+\beta+2 n+2)_2 ({\alpha_2}+\beta+2 n+2)_2
			({\alpha_2}+\beta+n+N+2)}
		\bigg(
		{\alpha_1}^2 ({\alpha_2}+\beta+n+N+2)\\
		+{\alpha_1}
		\left({\alpha_2}^2+{\alpha_2} (4 \beta+11 n+13)+3 \beta^2+\beta (13 n+2 N+17)+n (14 n+5
		N+37)+6 (N+4)\right)\\+{\alpha_2}^2 (\beta+n+N+2)
		+{\alpha_2} \left(3 \beta^2+\beta (13 n+2
		N+17)+n (14 n+5 N+37)+6 (N+4)\right)\\+13 \beta^2 n+2 \beta^3+17 \beta^2
		+N \left(2
		\left(5\beta n+\beta (\beta+6)+6 n^2\right)+29 n+17\right)\\+28 \beta n^2+74 \beta n+48 \beta+21
		n^3+83 n^2+108 n+46
		\bigg),\end{multlined}\\
		\det A(\vec{a},\vec{p})&=\dfrac{(n+1) (\alpha_1-\alpha_2+n+1) (\alpha_1+\beta+2n+2) (\alpha_2+\beta+2n+2)}{(N-2n-1) (\alpha_1+\beta+n+N+1)}.
	\end{align*}}
By following the same steps as in the previous case, we can conclude that the coefficients in \eqref{hola} are the same as those given by Theorem \ref{teoremaEbisuIwasaki}.
\end{proof}

Given these explicit expressions \eqref{aHahnmultiple} for $\{a_n\}$, one can check, factor by factor, that they are positive whenever $-1<\alpha_1-\alpha_2<0$. This follows from the fact that in the ratios of hypergeometric functions involved, both the numerator and denominator are positive.

These recursion coefficients in \eqref{recurrenciaHahnmultiple} are all positive if $-1<\alpha_1-\alpha_2<1$, region that determines a band in the parameter space. However, the PBF factorization only holds in the subregion $-1<\alpha_1-\alpha_2<0$ that we named the semi-band.  See the figure below:
\clearpage

\begin{center}
	\begin{tikzpicture}[arrowmark/.style 2 args={decoration={markings,mark=at position #1 with \arrow{#2}}},scale=1]
		\begin{axis}[axis lines=middle,axis equal,grid=both,xmin=-1, xmax=3,ymin=-1.5, ymax=3.5,
			xticklabel,yticklabel,disabledatascaling,xlabel=$\alpha_1$,ylabel=$\alpha_2$,every axis x label/.style={
				at={(ticklabel* cs:1)},
				anchor=south west,
			},
			every axis y label/.style={
				at={(ticklabel* cs:1.0)},
				anchor=south west,
			},grid style={line width=.1pt, draw=Bittersweet!10},
			major grid style={line width=.2pt,draw=Bittersweet!50},
			minor tick num=4,
			enlargelimits={abs=0.09},
			axis line style={latex'-latex'},Bittersweet]
			
			\draw [fill=DarkSlateBlue!20,opacity=.2,dashed,thick] (-1,4)--(-1,-1)--(5,-1)--(5,4)--(-1,4) ;
			\draw [fill=DarkSlateBlue!30,opacity=.5,dashed,thick] (-1,-1)--(-1,0)--(3,4) node[above, Black,sloped, pos=0.5] {$\alpha_1=\alpha_2+1$}--(5,5)--(5,4) --(0,-1) node[below, Black,sloped, pos=0.6] {$\alpha_1=\alpha_2-1$}--(-1,-1);
			\draw [
			pattern=north west lines, pattern color=DarkSlateBlue!50,opacity=.6,dashed] (-1,-1)--(0,-1)--(5,4)--(5,5)--(-1,-1) ;
			\draw [fill=DarkSlateBlue!30,opacity=.5,dashed,thick] (-1,-1) -- (4,4 )node[below, Black,sloped, pos=0.4] {$\alpha_1=\alpha_2$};
			\draw[thick,black] (axis cs:-1,0) circle[radius=2pt,opacity=0.2,fill]node[left,above ] {$-1$} ;
			\draw[thick,black] (axis cs:0,-1)circle[radius=2pt,opacity=0.2,fill] node[right,below ] {$-1$} ;
			\node[anchor = north east,Bittersweet] at (axis cs: 4.1,2.7) {$\mathscr R$} ;
			\node[anchor = north east,Bittersweet] at (axis cs: 0.1,.55) {$\mathscr S_+$} ;
			\node[anchor = north east,Bittersweet] at (axis cs: 0.7,0.05) {$\mathscr S_-$} ;
		\end{axis}
		\draw (3.5,-0.35) node
		{\begin{minipage}{0.5\textwidth}
				\begin{center}\small
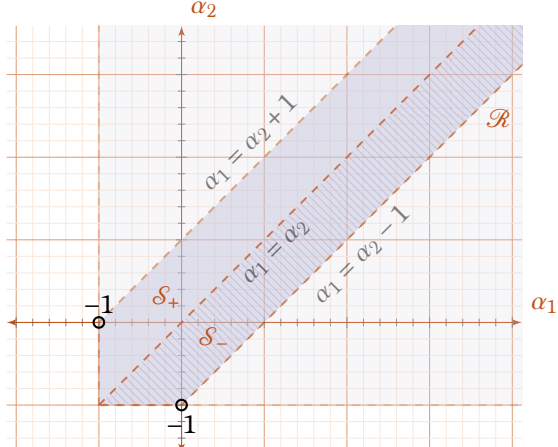
					\textbf{Allowed parameter region $\mathscr R$, nonnegativeness  band $\mathscr S_+\cup\mathscr S_-$ and PBF semi-band $\mathscr S_-$}
				\end{center}
		\end{minipage}};
	\end{tikzpicture}
\end{center}

\section{Bidiagonal factorization for the Hahn descendants in the multiple Askey scheme}
Now, examining the expressions \eqref{aHahnmultiple} for the Hahn $a_n$ coefficients and considering the established limit relations between the recurrence relation coefficients of the families in the Askey scheme \cite{ContinuosII, Arvesu, AskeyII, HahnI, Clasicos}, we can deduce explicit expressions for a bidiagonal factorization of all the Hahn descendants in the Askey scheme, except for the Hermite family. The limit relations leading to the Hermite family are still unknown in the literature.

\subsection{Jacobi--Piñeiro}

The coefficients of the recurrence relations, as shown in \eqref{eq:laT} and \eqref{eq:laT truncada}, can be obtained from Hahn \eqref{recurrenciaHahnmultiple} 
through the limit
\begin{align*}
   & \lim_{N\rightarrow\infty}\dfrac{b_n(\alpha_1,\alpha_2,\beta,N)}{N}, &
&\lim_{N\rightarrow\infty}\dfrac{c_n(\alpha_1,\alpha_2,\beta,N)}{N^2}, &&
\lim_{N\rightarrow\infty}\dfrac{d_n(\alpha_1,\alpha_2,\beta,N)}{N^3}.
\end{align*}
These ones are $b_n(\alpha_1,\alpha_2,\beta),c_n(\alpha_1,\alpha_2,\beta)$ and $d_n(\alpha_1,\alpha_2,\beta)$, cf. \cite[\S 3.1]{Clasicos}, \cite[\S 3.3]{ContinuosII} with:
\begin{align}
\label{recurrenciaJPmultiple}
\left\{\begin{aligned}
	 b_{2m}&=\dfrac{A(m,m,\alpha_1,\alpha_2,\beta,N)}{N+\alpha_1+\beta+m+1}+\dfrac{A(m,m,\alpha_2,\alpha_1+1,\beta,N)}{N+\alpha_2+\beta+m+1}+\dfrac{C(m+1,m+1,\alpha_1,\alpha_2,\beta,N)}{N-2m},\\
 b_{2m+1}&=\begin{multlined}[t][.8\textwidth]
 	\dfrac{A(m,m+1,\alpha_2,\alpha_1,\beta,N)}{N+\alpha_2+\beta+m+1}+\dfrac{A(m+1,m,\alpha_1,\alpha_2+1,\beta,N)}{N+\alpha_1+\beta+m+2}
 +\dfrac{C(m+1,m+2,\alpha_2,\alpha_1,\beta,N)}{N-2m-1},
 \end{multlined}\\
 c_{2m}&=\begin{multlined}[t][.8\textwidth]
 	\left(\dfrac{A(m,m,\alpha_1,\alpha_2,\beta,N)}{N+\alpha_1+\beta+m+1}+\dfrac{A(m,m,\alpha_2,\alpha_1+1,\beta,N)}{N+\alpha_2+\beta+m+1}\right)
\dfrac{C(m,m+1,\alpha_2,\alpha_1,\beta,N)}{N-2m+1}
\\
+\dfrac{A(m,m,\alpha_1,\alpha_2,\beta,N)B(m,m,\alpha_1,\alpha_2,\beta,N)}{(N+\alpha_1+\beta+m+1)(N-2m+1)}, \end{multlined}\\
c_{2m+1}&=\begin{multlined}[t][.8\textwidth]
	\left(\dfrac{A(m,m+1,\alpha_2,\alpha_1,\beta,N)}{N+\alpha_2+\beta+m+1}+\dfrac{A(m+1,m,\alpha_1,\alpha_2+1,\beta,N)}{N+\alpha_1+\beta+m+2}\right)\dfrac{C(m+1,m+1,\alpha_1,\alpha_2,\beta,N)}{N-2m}\\
+\dfrac{A(m,m+1,\alpha_2,\alpha_1,\beta,N)B(m,m+1,\alpha_2,\alpha_1,\beta,N)}{(N+\alpha_2+\beta+m+1)(N-2m)},
\end{multlined}\\
 d_{2m}&=\dfrac{A(m,m,\alpha_1,\alpha_2,\beta,N)B(m,m,\alpha_1,\alpha_2,\beta,N)C(m,m,\alpha_1,\alpha_2,\beta,N)}{(N+\alpha_1+\beta+m+1)(N-2m+1)(N-2m+2)},\\
 d_{2m+1}&=\dfrac{A(m,m+1,\alpha_2,\alpha_1,\beta,N)B(m,m+1,\alpha_2,\alpha_1,\beta,N)C(m,m+1,\alpha_2,\alpha_1,\beta,N)}{(N+\alpha_2+\beta+m+1)(N-2m)(N-2m+1)}.
\end{aligned}\right.
\end{align}
Being $A,B,C,D$ the functions defined in \eqref{coefsauxiliares}. These coefficients are all positive if $-1<\alpha_1-\alpha_2<1$.

\begin{pro}[Jacobi--Piñeiro bidiagonal factorization]
The entries $a_n (\alpha_1,\alpha_2,\beta)$ of a bidiagonal factorization \eqref{factorizacionbidiagonalmultiple} of the Jacobi-Piñeiro recursion matrix with coefficients given in \eqref{recurrenciaJPmultiple} are as follows:
\begin{align*}
	a^{}_{6n+1}&=\dfrac{(\alpha_1+1+n)(\alpha_1+\beta+2n+1)(\alpha_2+\beta+2n+1)}{(\alpha_1+\beta+3n+1)_2(\alpha_2+\beta+3n+1)},&
	a^{}_{6n+4}&=\dfrac{(\alpha_2+1+n)(\alpha_1+\beta+2n+2)(\alpha_2+\beta+2n+2)}{(\alpha_1+\beta+3n+3)(\alpha_2+\beta+3n+2)_2},\\
	a^{}_{6n+2}&=\dfrac{(\beta+2n+1)(\alpha_2-\alpha_1+n)(\alpha_2+\beta+2n+1)}{(\alpha_1+\beta+3n+2)(\alpha_2+\beta+3n+1)_2},&
	a^{}_{6n+5}&=\dfrac{(n+1)(\beta+2n+2)(\alpha_2+\beta+2n+2)}{(\alpha_1+\beta+3n+3)_2(\alpha_2+\beta+3n+3)},\\
	a^{}_{6n+3}&=\dfrac{(\beta+2n+1)(\alpha_1-\alpha_2+n+1)(\alpha_1+\beta+2n+2)}{(\alpha_1+\beta+3n+2)_2(\alpha_2+\beta+3n+2)},&
	a^{}_{6n+6}&=\dfrac{(n+1)(\beta+2n+2)(\alpha_1+\beta+2n+3)}{(\alpha_1+\beta+3n+4)(\alpha_2+\beta+3n+3)_2}.
\end{align*}
All of these entries are positive whenever $\alpha_1$ and $\alpha_2$ lies in the semi-band $-1<\alpha_1-\alpha_2<0$.
\end{pro}

\begin{proof}
The coefficients in \eqref{factorizacionbidiagonalmultiple} can be obtained from the Hahn coefficients in \eqref{aHahnmultiple} by taking the limit as $N$ approaches infinity of $\frac{a_n(\alpha_1, \alpha_2, \beta, N)}{N}$, and this yields the entries $a_n (\alpha_1, \alpha_2, \beta)$.
\end{proof}

\subsection{Multiple Meixner of the first kind}

The coefficients of the recurrence relations, as shown in \eqref{eq:laT} and \eqref{eq:laT truncada}, can be obtained from Hahn \eqref{recurrenciaHahnmultiple} 
through the limit
\begin{align*}
&\lim_{N\rightarrow\infty} b_n\left(c_1N,c_2N,-N,-\beta\right), &
&\lim_{N\rightarrow\infty} c_n\left(c_1N,c_2N,-N,-\beta\right), &
&\lim_{N\rightarrow\infty} d_n\left(c_1N,c_2N,-N,-\beta\right).
\end{align*}
These ones are, cf. \cite[\S 4.2]{Arvesu}, $b_n(\beta,c_1,c_2),c_n(\beta,c_1,c_2)$ and $d_n(\beta,c_1,c_2)$ with:
\begin{align}
\label{recurrenciaMeixnerFKmultiple}
 \left\{\begin{aligned}
 	   b_{2m}&=m\dfrac{1+c_1}{1-c_1}+m\left(\dfrac{c_1}{1-c_1}+\dfrac{c_2}{1-c_2}+1\right)+\dfrac{c_1}{1-c_1}\beta, &
    b_{2m+1}&=m\dfrac{1+c_2}{1-c_2}+(m+1)\left(\dfrac{c_1}{1-c_1}+\dfrac{c_2}{1-c_2}+1\right)+\dfrac{c_2}{1-c_2}\beta,\\
    c_{2m}&=(\beta+2m-1)\left(m\dfrac{c_1}{(1-c_1)^2}+m\dfrac{c_2}{(1-c_2)^2}\right), &
    c_{2m+1}&=(\beta+2m)\left((m+1)\dfrac{c_1}{(1-c_1)^2}+m\dfrac{c_2}{(1-c_2)^2}\right),\\
    d_{2m}&=\dfrac{m(\beta+2m-2)(\beta+2m-1)c_1(c_1-c_2)}{(1-c_1)^3(1-c_2)}, &
    d_{2m+1}&=\dfrac{m(\beta+2m-1)(\beta+2m)c_2(c_2-c_1)}{(1-c_1)(1-c_2)^3}.
 \end{aligned}\right.
\end{align}

\begin{pro}[Multiple Meixner of the first kind  bidiagonal factorization]
	The entries $a_n (\beta,c_1,c_2)$ of a bidiagonal factorization \eqref{factorizacionbidiagonalmultiple} of the Meixner of the first kind recursion matrix with coefficients given in \eqref{recurrenciaMeixnerFKmultiple} are as follows:
	\begin{align*}
		a^{}_{6n+1}&=\dfrac{(\beta+2n)c_1}{1-c_1}, &
		a^{}_{6n+4}&=\dfrac{(\beta+2n+1)c_2}{1-c_2},\\
		a^{}_{6n+2}&=    \dfrac{(n)_n(\beta+2n)(c_2-c_1)}{(n+1)_n(1-c_1)(1-c_2)}\dfrac{\pFq{2}{1}{-n,\beta}{-2n+1}{\frac{c_1-c_2}{1-c_2}}}{\pFq{2}{1}{-n,\beta}{-2n}{\frac{c_1-c_2}{1-c_2}}},
		&a^{}_{6n+5}&=\dfrac{(n+1)(\beta+2n+1)(c_1-c_2)}{(2n+1)(1-c_1)(1-c_2)}\dfrac{\pFq{2}{1}{-n,\beta}{-2n}{\frac{c_1-c_2}{1-c_2}}}{\pFq{2}{1}{-n-1,\beta}{-2n-1}{\frac{c_1-c_2}{1-c_2}}},\\
		a^{}_{6n+3}&=\dfrac{2n+1}{1-c_1}\dfrac{\pFq{2}{1}{-n-1,\beta}{-2n-1}{\frac{c_1-c_2}{1-c_2}}}{\pFq{2}{1}{-n,\beta}{-2n}{\frac{c_1-c_2}{1-c_2}}}, &
		a^{}_{6n+6}&=\dfrac{2(n+1)}{1-c_2}\dfrac{\pFq{2}{1}{-n-1,\beta}{-2n-2}{\frac{c_1-c_2}{1-c_2}}}{\pFq{2}{1}{-n-1,\beta}{-2n-1}{\frac{c_1-c_2}{1-c_2}}}.
	\end{align*}
\end{pro}
\begin{proof}
The coefficients \eqref{factorizacionbidiagonalmultiple}  can be obtained from Hahn \eqref{aHahnmultiple}
through $\lim\limits_{N\rightarrow\infty}a_n(c_1N,c_2N,-N,-\beta)$.
\end{proof}

\begin{rem}
	These ones can never be all positive because of the factor $(c_2-c_1)$.
\end{rem}

\subsection{Multiple Meixner of the second kind}

The coefficients of the recurrence relations, as shown in \eqref{eq:laT} and \eqref{eq:laT truncada}, can be obtained from Hahn \eqref{recurrenciaHahnmultiple} 
through the limit
\begin{align*}    &\lim_{N\rightarrow\infty}b_n\left(\beta_1-1,\beta_2-1,\frac{1-c}{c}N,N\right), &\lim_{N\rightarrow\infty}c_n\left(\beta_1-1,\beta_2-1,\frac{1-c}{c}N,N\right), & &\lim_{N\rightarrow\infty}d_n\left(\beta_1-1,\beta_2-1,\frac{1-c}{c}N,N\right).
\end{align*}
The corresponding recurrence coefficients  $b_{m}(\beta_1,\beta_2,c)$, $c_{m}(\beta_1,\beta_2,c)$  and $d_{m}(\beta_1,\beta_2,c)$  are, cf. \cite[\S 4.3]{Arvesu}
\begin{align}
\label{recurrenciaMeixnerSKmultiple}
\left\{\begin{aligned}
	 b_{2m}&=2m+\dfrac{c}{1-c}(\beta_1+3m),&
 b_{2m+1}&=2m+1+\dfrac{c}{1-c}(\beta_2+3m+1),\\
 c_{2m}&=\dfrac{c}{(1-c)^2}m(\beta_1+\beta_2+3m-2),&
 c_{2m+1}&=\dfrac{c}{(1-c)^2}((m+1)\beta_1+m(\beta_2+3m+1)),\\
 d_{2m}&=\dfrac{c^2}{(1-c)^3}m(m+\beta_1-1)(m+\beta_1-\beta_2),&
 d_{2m+1}&=\dfrac{c^2}{(1-c)^3}m(m+\beta_2-1)(m+\beta_2-\beta_1),
\end{aligned}\right.
\end{align}
which are all positive if $-1<\beta_1-\beta_2<1$.

\begin{pro}[Multiple Meixner of the second kind  bidiagonal factorization]
The entries $a_n (\beta_1,\beta_2,c)$ of a bidiagonal factorization \eqref{factorizacionbidiagonalmultiple} of the Meixner of the first kind recursion matrix with coefficients given in \eqref{recurrenciaMeixnerSKmultiple} are as follows:
\begin{align*}
	a^{}_{6n+1}&=\dfrac{(\beta_1+n)c}{1-c}, &
	a^{}_{6n+4}&=\dfrac{(\beta_2+n)c}{1-c},\\
	a^{}_{6n+2}&=\dfrac{(n)_n(\beta_2-\beta_1+n)c}{(n+1)_n(1-c)}\dfrac{\pFq{2}{1}{-n,\beta_2-\beta_1-n}{-2n+1}{\frac{c}{c-1}}}{\pFq{2}{1}{-n,\beta_2-\beta_1-n}{-2n}{\frac{c}{c-1}}},&
	a^{}_{6n+5}&=\dfrac{(n+1)(\beta_1-\beta_2+n+1)c}{(2n+1)(1-c)^2}\dfrac{\pFq{2}{1}{-n,\beta_2-\beta_1-n}{-2n}{\frac{c}{c-1}}}{\pFq{2}{1}{-n-1,\beta_2-\beta_1-n-1}{-2n-1}{\frac{c}{c-1}}},\\
	a^{}_{6n+3}=&(2n+1)
	\dfrac{\pFq{2}{1}{-n-1,\beta_2-\beta_1-n-1}{-2n-1}{\frac{c}{c-1}}}{\pFq{2}{1}{-n,\beta_2-\beta_1-n}{-2n}{\frac{c}{c-1}}}, &
	a^{}_{6n+6}=&\dfrac{2(n+1)}{1-c}
	\dfrac{\pFq{2}{1}{-n-1,\beta_2-\beta_1-n-1}{-2n-2}{\frac{c}{c-1}}}{\pFq{2}{1}{-n-1,\beta_2-\beta_1-n-1}{-2n-1}{\frac{c}{c-1}}}.
\end{align*}
These coefficients  are all positive in the semi-band  $-1<\beta_1-\beta_2<0$.
\end{pro}

\begin{proof}
The entries $a_n(\beta_1,\beta_2,c)$ in  \eqref{factorizacionbidiagonalmultiple}  can be obtained from Hahn \eqref{aHahnmultiple}
through
$    \lim\limits_{N\rightarrow\infty}a_n\left(\beta_1-1,\beta_2-1,\frac{1-c}{c}N,N\right)$.
\end{proof}

\subsection{Multiple Kravchuk}

The coefficients of the recurrence relations, as shown in \eqref{eq:laT} and \eqref{eq:laT truncada}, can be obtained from Hahn \eqref{recurrenciaHahnmultiple} 
through the limit
\begin{align*}
&\lim_{t\rightarrow\infty} b_n\left(\dfrac{p_1}{1-p_1}t,\dfrac{p_2}{1-p_2}t,t,N\right),&
&\lim_{t\rightarrow\infty} c_n\left(\dfrac{p_1}{1-p_1}t,\dfrac{p_2}{1-p_2}t,t,N\right),&
&\lim_{t\rightarrow\infty} d_n\left(\dfrac{p_1}{1-p_1}t,\dfrac{p_2}{1-p_2}t,t,N\right)
\end{align*}
These recursion coefficients  $b_n(p_1,p_2,N)$, $c_n(p_1,p_2,N)$ and $d_n(p_1,p_2,N)$, are, cf. \cite[\S 4.4]{Arvesu},
\begin{align}
\label{recurrenciaKravchukmultiple}
  b_{2m}&=2m+(N-3m)p_1-mp_2, &
   b_{2m+1}&=2m+1-(m+1)p_1+(N-3m-1)p_2,\\
    c_{2m}&=(N-2m+1)m(p_1(1-p_1)+p_2(1-p_2)), &
    c_{2m+1}&=(N-2m)\left((m+1)p_1(1-p_1)+mp_2(1-p_2)\right),\\
    d_{2m}&=(N-2m+1)(N-2m+2)mp_1(1-p_1)(p_1-p_2),&
    d_{2m+1}&=(N-2m)(N-2m+1)mp_2(1-p_2)(p_2-p_1).
\end{align}

\begin{pro}[Multiple Kravchuk  bidiagonal factorization]
	The entries $a_n (p_1,p_2,N)$ of a bidiagonal factorization \eqref{factorizacionbidiagonalmultiple} of the multiple Kravchuk recursion matrix with coefficients given in \eqref{recurrenciaKravchukmultiple} are as follows:
\begin{align*}
	a^{}_{6n+1}&=(N-2n)p_1, & a^{}_{6n+4}&=(N-2n-1)p_2,\\
	a^{}_{6n+2}&=\dfrac{(n)_n(N-2n)(p_2-p_1)}{(n+1)_n}
	\dfrac{\pFq{2}{1}{-n,-N}{-2n+1}{\frac{p_2-p_1}{1-p_1}}}{\pFq{2}{1}{-n,-N}{-2n}{\frac{p_2-p_1}{1-p_1}}}, &
	a^{}_{6n+5}&=\dfrac{(n+1)(N-2n-1)(p_1-p_2)}{(2n+1)}
	\dfrac{\pFq{2}{1}{-n,-N}{-2n}{\frac{p_2-p_1}{1-p_1}}}{\pFq{2}{1}{-n-1,-N}{-2n-1}{\frac{p_2-p_1}{1-p_1}}},\\
	a^{}_{6n+3}&=(2n+1)(1-p_1)\dfrac{\pFq{2}{1}{-n-1,-N}{-2n-1}{\frac{p_2-p_1}{1-p_1}}}{\pFq{2}{1}{-n,-N}{-2n}{\frac{p_2-p_1}{1-p_1}}}, &
	a^{}_{6n+6}&=
	2(n+1)(1-p_2)\dfrac{\pFq{2}{1}{-n-1,-N}{-2n-2}{\frac{p_2-p_1}{1-p_1}}}{\pFq{2}{1}{-n-1,-N}{-2n-1}{\frac{p_2-p_1}{1-p_1}}}.
\end{align*}
\end{pro}

\begin{proof}
	The \eqref{factorizacionbidiagonalmultiple} coefficients can be obtained from Hahn \eqref{aHahnmultiple}
through
$  \lim\limits _{t\rightarrow\infty}a_n\left(\frac{p_1}{1-p_1}t,\frac{p_2}{1-p_2}t,t,N\right)$. 
\end{proof}

These ones can never be all positive because of the factor $(p_2-p_1)$.

\subsection{Multiple Laguerre of first kind}

The recurrence coefficients \eqref{eq:laT}, \eqref{eq:laT truncada} can be obtained indistictively from Jacobi--Piñeiro \eqref{recurrenciaJPmultiple} and Meixner of second kind \eqref{recurrenciaMeixnerSKmultiple} respectively through
\begin{align*}
    &\lim_{\beta\rightarrow\infty}\beta b_n(\alpha_1,\alpha_2,\beta), &&\lim_{\beta\rightarrow\infty}\beta^2c_n(\alpha_1,\alpha_2,\beta),
    &&\lim_{\beta\rightarrow\infty}\beta^3d_n(\alpha_1,\alpha_2,\beta),\\
   & \lim_{c\rightarrow1}(1-c)b_n(\alpha_1+1,\alpha_2+1,c), &&\lim_{c\rightarrow1}{(1-c)^2}c_n(\alpha_1+1,\alpha_2+1,c), &&\lim_{c\rightarrow1}{(1-c)^3}d_n(\alpha_1+1,\alpha_2+1,c).
\end{align*}
The recursion  coefficients $b_n(\alpha_1,\alpha_2)$, $c_n(\alpha_1,\alpha_2)$ and $d_n(\alpha_1,\alpha_2)$ are \cite[\S 3.2]{Clasicos}
\begin{align*}
 b_{2m}&=3m+1+\alpha_1,&
 b_{2m+1}&=3m+2+\alpha_2,\\
 c_{2m}&=m(3m+\alpha_1+\alpha_2),&
 c_{2m+1}&=3m^2+m(\alpha_1+\alpha_2+3)+\alpha_1+1,\\
 d_{2m}&=m(m+\alpha_1)(m+\alpha_1-\alpha_2),
 &
 d_{2m+1}&=m(m+\alpha_2)(m+\alpha_2-\alpha_1),
\end{align*}
which are all positive if $-1<\alpha_1-\alpha_2<1$.

\begin{pro}[Multiple Laguerre of first kind  bidiagonal factorization]
	The entries $a_n (\alpha_1,\alpha_2)$ of a bidiagonal factorization \eqref{factorizacionbidiagonalmultiple} of the  Laguerre of first kind recursion matrix with coefficients given in \eqref{recurrenciaKravchukmultiple} are as follows:
\begin{align}
	a^{}_{6n+1}&=\alpha_1+1+n,&
	a^{}_{6n+4}&=\alpha_2+1+n,\\
	a^{}_{6n+2}&=\alpha_2-\alpha_1+n,&a^{}_{6n+5}&=n+1,\\
	a^{}_{6n+3}&=\alpha_1-\alpha_2+n+1,&
	a^{}_{6n+6}&=n+1.
\end{align}
These coefficients are all positive whenever $-1<\alpha_1-\alpha_2<0$.
\end{pro}

\begin{proof}
	Coefficients in \eqref{factorizacionbidiagonalmultiple}  can be obtained from Jacobi--Piñeiro and Meixner of second kind
through the respective limits
$   \lim\limits_{\beta\rightarrow\infty}\beta a_n(\alpha_1,\alpha_2,\beta)$ or 
  $  \lim\limits_{c\rightarrow1}(1-c)a_n(\alpha_1+1,\alpha_2+1,c)$.
\end{proof}

\subsection{Multiple Laguerre of second kind}

The recurrence coefficients \eqref{eq:laT}, \eqref{eq:laT truncada} can be obtained indistictively from Jacobi--Piñeiro \eqref{recurrenciaJPmultiple} and Meixner of first kind \eqref{recurrenciaMeixnerFKmultiple}, respectively, through
\begin{align*}
    &\lim_{t\rightarrow\infty}(-t)\left(b_{n}\left(c_1t,c_2t,\alpha_0+1\right)-1\right),&
        &\lim_{t\rightarrow\infty}(-t)^2c_{n}\left(c_1t,c_2t,\alpha_0+1\right),&
         &\lim_{t\rightarrow\infty}\left(-t\right)^3d_{n}\left(c_1t,c_2t,\alpha_0+1\right),\\
    &\lim_{t\rightarrow\infty}\dfrac{1}{t}b_n\left(\alpha_0+1,\dfrac{t}{t+c_1},\dfrac{t}{t+c_2}\right), &    
    &\lim_{t\rightarrow\infty}\dfrac{1}{t^2}c_n\left(\alpha_0+1,\dfrac{t}{t+c_1},\dfrac{t}{t+c_2}\right), &
    &\lim_{t\rightarrow\infty}\dfrac{1}{t^3}d_n\left(\alpha_0+1,\dfrac{t}{t+c_1},\dfrac{t}{t+c_2}\right).
\end{align*}
These coefficients $b_n(\alpha_0,c_1,c_2)$, $c_n(\alpha_0,c_1,c_2)$ and $d_n(\alpha_0,c_1,c_2)$ are, cf. \cite[\S 3.3]{Clasicos},
\begin{align}\label{recurrenciaLaguerreSKmultiple}
  \left\{  \begin{aligned}
  	b_{2m}&=\dfrac{m(c_1+3c_2)+c_2+\alpha_0 c_2}{c_1c_2}, &
    b_{2m+1}&=\dfrac{m(3c_1+c_2)+2c_1+c_2+\alpha_0 c_1}{c_1c_2},\\
    c_{2m}&=\dfrac{(2m+\alpha_0)m(c^2_1+c^2_2)}{c^2_1c^2_2}, &
    c_{2m+1}&=\dfrac{(2m+1+\alpha_0)\left(mc_1^2+(m+1)c_2^2\right)}{c^2_1c^2_2},\\
    d_{2m}&=\dfrac{m(2m+\alpha_0)(2m+\alpha_0-1)(c_2-c_1)}{c^3_1c_2}, &
    d_{2m+1}&=\dfrac{m(2m+\alpha_0)(2m+\alpha_0+1)(c_1-c_2)}{c_1c^3_2}.
  \end{aligned}\right.
\end{align}

\begin{pro}[Multiple Laguerre of second kind  bidiagonal factorization]
	The entries $a_n (\alpha_1,\alpha_2)$ of a bidiagonal factorization \eqref{factorizacionbidiagonalmultiple} of the  Laguerre of first kind recursion matrix with coefficients given in \eqref{recurrenciaLaguerreSKmultiple} are as follows:
	\begin{align*}
		a^{}_{6n+1}&=\dfrac{\alpha_0+2n+1}{c_1}&
		a^{}_{6n+4}&=\dfrac{\alpha_0+2n+2}{c_2},\\
		a^{}_{6n+2}&=\dfrac{(n)_n(\alpha_0+2n+1)(c_1-c_2)}{(n+1)_nc_1c_2}\dfrac{\pFq{2}{1}{-n,\alpha_0+1}{-2n+1}{\frac{c_2-c_1}{c_2}}}{\pFq{2}{1}{-n,\alpha_0+1}{-2n}{\frac{c_2-c_1}{c_2}}},&a^{}_{6n+5}&= \dfrac{(n+1)(\alpha_0+2n+2)(c_2-c_1)}{(2n+1)c_1c_2}\dfrac{\pFq{2}{1}{-n,\alpha_0+1}{-2n}{\frac{c_2-c_1}{c_2}}}{\pFq{2}{1}{-n-1,\alpha_0+1}{-2n-1}{\frac{c_2-c_1}{c_2}}},\\
		a^{}_{6n+3}&=\dfrac{2n+1}{c_1}\dfrac{\pFq{2}{1}{-n-1,\alpha_0+1}{-2n-1}{\frac{c_2-c_1}{c_2}}}{\pFq{2}{1}{-n,\alpha_0+1}{-2n}{\frac{c_2-c_1}{c_2}}},&
		a^{}_{6n+6}&=\dfrac{2(n+1)}{c_2}\dfrac{\pFq{2}{1}{-n-1,\alpha_0+1}{-2n-2}{\frac{c_2-c_1}{c_2}}}{\pFq{2}{1}{-n-1,\alpha_0+1}{-2n-1}{\frac{c_2-c_1}{c_2}}}.
	\end{align*}
\end{pro}
\begin{proof}
	
The \eqref{factorizacionbidiagonalmultiple} coefficients can be obtained from Meixner  of first kind
through
$  \lim\limits_{t\rightarrow\infty}\dfrac{1}{t}a_n\left(\alpha_0+1,\dfrac{t}{t+c_1},\dfrac{t}{t+c_2}\right)$.
\end{proof}

These ones can never be all positive becasuse of the factor $(c_2-c_1)$.

\subsection{Multiple Charlier}

The recurrence coefficients \eqref{eq:laT}, \eqref{eq:laT truncada} can be obtained indistinctively from Meixner of first kind \eqref{recurrenciaMeixnerFKmultiple} and Kravchuk \eqref{recurrenciaKravchukmultiple} respectively through 
\begin{align*}
    &\lim_{\beta\rightarrow\infty} b_n\left(\beta,\dfrac{b_1}{\beta},\dfrac{b_2}{\beta}\right),&
     &\lim_{\beta\rightarrow\infty} c_n\left(\beta,\dfrac{b_1}{\beta},\dfrac{b_2}{\beta}\right),&
      &\lim_{\beta\rightarrow\infty} d_n\left(\beta,\dfrac{b_1}{\beta},\dfrac{b_2}{\beta}\right),\\
   & \lim_{N\rightarrow\infty} b_n\left(\dfrac{b_1}{N},\dfrac{b_2}{N},N\right), &&
        \lim_{N\rightarrow\infty} c_n\left(\dfrac{b_1}{N},\dfrac{b_2}{N},N\right), &&\lim_{N\rightarrow\infty} d_n\left(\dfrac{b_1}{N},\dfrac{b_2}{N},N\right).
\end{align*}
These  coefficients  $b_n(b_1,b_2)$, $c_n(b_1,b_2)$ and $d_n(b_1,b_2)$, are, cf. \cite[\S 4.1]{Arvesu}
\begin{align}\label{recurrenciaCharliermultiple}
\left\{ \begin{aligned}
 	   b_{2m}&=2m+b_1,&    b_{2m+1}&=2m+1+b_2,\\
    c_{2m}&=m(b_1+b_2), &
    c_{2m+1}&=(m+1)b_1+mb_2,\\
    d_{2m}&=mb_1(b_1-b_2), &
    d_{2m+1}&=mb_2(b_2-b_1).
 \end{aligned}\right.
\end{align}

\begin{pro}[Multiple Charlier of second kind  bidiagonal factorization]
	The entries $a_n (b_1,b_2)$ of a bidiagonal factorization \eqref{factorizacionbidiagonalmultiple} of the  multiple Charlier recursion matrix with coefficients given in \eqref{recurrenciaCharliermultiple} are as follows:
	\begin{align*}
		a^{}_{6n+1}&=b_1,&
		a^{}_{6n+4}&=b_2,\\
		a^{}_{6n+2}&=\dfrac{(n)_n(b_2-b_1)}{(n+1)_n}
		\dfrac{\pFq{1}{1}{-n}{-2n+1}{b_1-b_2}}{\pFq{1}{1}{-n}{-2n}{b_1-b_2}},
		&a^{}_{6n+5}&= \dfrac{(n+1)(b_1-b_2)}{(2n+1)}
		\dfrac{\pFq{1}{1}{-n}{-2n}{b_1-b_2}}{\pFq{1}{1}{-n-1}{-2n-1}{b_1-b_2}},\\
		a^{}_{6n+3}&=(2n+1)\dfrac{\pFq{1}{1}{-n-1}{-2n-1}{b_1-b_2}}{\pFq{1}{1}{-n}{-2n}{b_1-b_2}},&
		a^{}_{6n+6}&=2(n+1)\dfrac{\pFq{1}{1}{-n-1}{-2n-2}{b_1-b_2}}{\pFq{1}{1}{-n-1}{-2n-1}{b_1-b_2}}.
	\end{align*}
\end{pro}
\begin{proof}
	
	The  coefficients in \eqref{factorizacionbidiagonalmultiple} can be obtained from Meixner  of first  kind or Kravchuk
	through
	$ \lim\limits_{\beta\rightarrow\infty}a_n\left(\beta,\frac{b_1}{\beta},\frac{b_2}{\beta}\right)$ and $\lim\limits_{N\rightarrow\infty}a_n\left(\frac{b_1}{N},\frac{b_2}{N},N\right)$.
\end{proof}

These ones can never be all positive because of the factor $(b_2-b_1)$.

\section*{Conclusions and outlook}

This paper explores a factorization using bidiagonal matrices of the recurrence matrix of Hahn multiple orthogonal polynomials. These factorizations naturally arise from the theory of multiple orthogonal polynomials. Specifically, there is a natural choice of the initial parameter for the type I multiple orthogonal polynomials, which leads to this bidiagonal factorization. With this parameter choice, we obtain closed-form expressions for these factorizations.

The fact that the recurrence matrix becomes a nonnegative matrix in a specific region of the parameter space makes these positive bidiagonal factorizations (PBFs) relevant for finding associated Markov chains. These Markov chains have stochastic matrices that can be factorized in terms of stochastic matrices describing pure birth or pure death Markov chains.

In this work, tetradiagonal Hessenberg matrices are presented, but there is potential to go further by considering more than two measures for the multiple orthogonal polynomials, thus obtaining more general banded Hessenberg matrices.

For future research, we can explore the connection of this bidiagonal factorization to the theory of mixed orthogonal polynomials. In this case, we would consider banded matrices that represent their recurrence relations.

\section*{Acknowledgments}

AB acknowledges the Centro de Matemática da Universidade de Coimbra, UIDB/00324/2020, which is funded by the Portuguese Government through FCT/MECS.

JEFD acknowledges the CIDMA Center for Research and Development in Mathematics and Applications (University of Aveiro) and the Portuguese Foundation for Science and Technology (FCT) for their support within projects UIDB/04106/2020 and UIDP/04106/2020, as well as [PID2021-122154NB-I00], titled ``Ortogonalidad y Aproximación con Aplicaciones en Machine Learning y Teoría de la Probabilidad.'' Additionally, he acknowledges the PhD contract UI/BD/152576/2022 from FCT Portugal.

AF acknowledges the CIDMA Center for Research and Development in Mathematics and Applications (University of Aveiro) and the Portuguese Foundation for Science and Technology (FCT) for their support within projects UIDB/04106/2020 and UIDP/04106/2020.

MM acknowledges the Spanish ``Agencia Estatal de Investigación'' research projects [PGC2018-096504-B-C33], titled ``Ortogonalidad y Aproximación: Teoría y Aplicaciones en Física Matemática,'' and [PID2021-122154NB-I00], titled ``Ortogonalidad y Aproximación con Aplicaciones en Machine Learning y Teoría de la Probabilidad.''

\section*{Declarations}

\begin{enumerate}
	\item \textbf{Conflict of interest:} The authors declare no conflict of interest.
	\item \textbf{Ethical approval:} Not applicable.
	\item \textbf{Contributions:} All the authors have contribute equally.
	\item \textbf{Generative AI and AI-assisted technologies in the writing process:}
	During the preparation of this work the authors used ChatGPT  in order to improve English grammar, syntax, spelling and wording. After using this tool/service, the authors reviewed and edited the content as needed and take full responsibility for the content of the publication.
	\item \textbf{Data availability:} This paper has associated data. There are is Mathematica notebook: 
\begin{itemize}
	\item 	 \texttt{BidiagonalFactorizationMultipleOrthogonalPolyomials.nb}.
\end{itemize}
that has been uploaded to the  \hyperref{https://notebookarchive.org}{}{}{Mathematica Notebook Archive} and to  \hyperref{https://github.com/ManuelManas/Bidiagonal_factorization/blob/main/}{}{}{GitHub}.
\end{enumerate}

\end{document}